\documentclass[12pt,twoside,reqno]{amsart}
\usepackage{amsmath}
\usepackage{amsfonts}
\usepackage{amssymb}
\usepackage{color}
\usepackage{mathrsfs}
\usepackage{cite}
\usepackage{geometry}
\usepackage{marginnote}
\usepackage{todonotes}
\allowdisplaybreaks
\newtheorem{theorem}{Theorem}[section]
\newtheorem{corollary}[theorem]{Corollary}
\newtheorem{lemma}[theorem]{Lemma}
\newtheorem{proposition}[theorem]{Proposition}
\newtheorem{definition}[theorem]{Definition}
\newtheorem{remark}[theorem]{Remark}
\allowdisplaybreaks
\numberwithin{equation}{section}
\begin{document}
\title{Long term spatial homogeneity for a chemotaxis model with local sensing and consumption} 

\author{Philippe Lauren\c{c}ot}
\address{Laboratoire de Math\'ematiques (LAMA) UMR~5127, Universit\'e Savoie Mont Blanc, CNRS\\
F--73000 Chamb\'ery, France}
\email{philippe.laurencot@univ-smb.fr}


\date{\today}

\begin{abstract}
Global weak solutions to a chemotaxis model with local sensing and consumption are shown to converge to spatially homogeneous steady states in the large time limit, when the motility is assumed to be positive and $C^1$-smooth on $[0,\infty)$. The result is valid in arbitrary space dimension $n\ge 1$ and extends a previous result which only deals with space dimensions $n\in \{1,2,3\}$. 
\end{abstract}

\maketitle

%
%
\pagestyle{myheadings}
\markboth{\sc{Ph. Lauren\c cot}}{\sc{Long term spatial homogeneity for a chemotaxis model}}

\section{Introduction}\label{sec1}

Let $\Omega$ be a smooth bounded domain of $\mathbb{R}^n$, $n\ge 1$, and consider the initial boundary value problem
\begin{subequations}\label{ks}
	\begin{align}
		\partial_t u & = \Delta (u\gamma(v)) \;\;\text{ in }\;\; (0,\infty)\times \Omega\,, \label{ks1} \\
		\partial_t v & = \Delta v - uv \;\;\text{ in }\;\; (0,\infty)\times \Omega\,, \label{ks2} \\
		\nabla (u\gamma(v))\cdot \mathbf{n} & = \nabla v\cdot \mathbf{n} = 0 \;\;\text{ on }\;\; (0,\infty)\times \partial\Omega\,, \label{ks3} \\
		(u,v)(0) & = (u^{in},v^{in}) \;\;\text{ in }\;\; \Omega\,, \label{ks4} 
	\end{align}
\end{subequations}
which describes the dynamics of a population of bacteria with non-negative density~$u$ and of a signal with non-negative concentration~$v$. On the one hand, according to~\eqref{ks1}, the diffusive motion of the bacteria is not only monitored by the signal through the motility function $\gamma$ but also biased by a chemotactic effect generated by the signal. On the other hand, the signal is consumed by the bacteria, as reflected by the reaction term on the right hand side of~\eqref{ks2}. The latter mechanism is in sharp contrast with classical Keller-Segel chemotaxis models \cite{KeSe1970}, in which the sink term $-uv$ in~\eqref{ks2} is replaced by $u-v$, so that bacteria produce the signal that alters their motion, see the survey articles \cite{BBTW2015, HiPa2009, Hors2003, Pain2019} and the references therein for a more precise account. Therefore, the dynamics of~\eqref{ks} is expected to differ significantly. A first hint in that direction is the following property: if $(u_s,v_s)$ is a stationary solution to~\eqref{ks} with $u_s\not\equiv 0$, then necessarily $v_s\equiv 0$ by~\eqref{ks2}. In that case, it readily follows from~\eqref{ks1} that $\gamma(0)u_s=\mathrm{const.}$, which reduces to $u_s=\mathrm{const.}$ when $\gamma(0)>0$. It is thus expected that the positivity of both $\gamma(0)$ and $\|u^{in}\|_1$ implies that any global non-negative solution $(u,v)$ to~\eqref{ks} satisfies
\begin{equation}
	\lim_{t\to\infty} (u(t),v(t)) = \left( \frac{\|u^{in}\|_1}{|\Omega|} , 0 \right) \label{ltcv}
\end{equation}
in an appropriate topology. That this convergence holds true in $L^\infty(\Omega,\mathbb{R}^2)$ is shown in~\cite[Theorem~1.2]{LiZh2021} when $\gamma\in C^3([0,\infty))$ is positive on $[0,\infty)$ with $\gamma'<0$ on $(0,\infty)$ and $\|v^{in}\|_\infty$ is sufficiently small and in~\cite{LiWi2023b} when $\gamma\in C^3([0,\infty))$ is positive on $[0,\infty)$ and the space dimension~$n$ ranges in $\{1,2,3\}$. The required regularity of~$\gamma$ is subsequently relaxed in~\cite[Theorem~1.2]{LiWi2023a}, where the validity of~\eqref{ltcv} is established under the sole assumption 
\begin{equation}
	\gamma\in C^1([0,\infty))\,, \quad \gamma>0 \;\;\text{ on }\;\; [0,\infty)\,, \label{n3}
\end{equation}
still for $n\in \{1,2,3\}$, though in the weaker topology $H^1(\Omega)'\times L^\infty(\Omega)$. The main purpose of this note is to show that the assumption~\eqref{n3} is sufficient to prove that the convergence~\eqref{ltcv} holds true in arbitrary space dimension in $H^1(\Omega)'\times H^1(\Omega)$, see Theorem~\ref{thm1} below. When $n\in \{1,2,3\}$, we further deduce the convergence of~$v$ to zero in $L^\infty(\Omega)$ from Theorem~\ref{thm1}, the continuous embedding of $H^2(\Omega)$, and the time monotonicity of $\|v\|_\infty$, see Corollary~\ref{cor1} below.

The statement of the main result of this note requires to introduce some notation: first, for $z\in H^1(\Omega)'$, we set $\langle z\rangle := \langle z , 1 \rangle_{(H^1)',H^1}/|\Omega|$ and note that
\begin{equation*}
	\langle z\rangle = \frac{1}{|\Omega|} \int_\Omega z(x)\ \mathrm{d}x \;\;\text{ for }\;\; z\in H^1(\Omega)'\cap L^1(\Omega).
\end{equation*}
Next, for $z\in H^1(\Omega)'$ with $\langle z \rangle = 0$, let $\mathcal{K}[z]\in H^1(\Omega)$ be the unique (variational) solution to 
\begin{subequations}\label{n2}
	\begin{equation}
		-\Delta\mathcal{K}[z] = z \;\;\text{ in }\;\; \Omega\,, \qquad \nabla\mathcal{K}[z]\cdot \mathbf{n} = 0 \;\;\text{ on }\;\; \partial\Omega\,, \label{n2a}
	\end{equation}
satisfying
	\begin{equation}
		\langle\mathcal{K}[z] \rangle = 0\,. \label{n2b}
	\end{equation}
\end{subequations}
Also, for $p\in [1,\infty]$, we denote the positive cone of $L^p(\Omega)$ by $L_+^p(\Omega)$.

\begin{theorem}\label{thm1}
	Assume that $\gamma$ satisfies~\eqref{n3} and consider $u^{in}\in L_+^1(\Omega)\cap H^1(\Omega)'$ and $v^{in} \in L_+^\infty(\Omega)\cap H^1(\Omega)$ with $M:=\langle u^{in} \rangle>0$. If $(u,v)$ is a global weak solution to~\eqref{ks} in the sense of Definition~\ref{def2} below, then
	\begin{align}
		& \lim_{t\to \infty} \|\nabla P(t)\|_2 = \lim_{t\to\infty} \|v(t)\|_{H^1} = 0\,, \label{cvp} \\
		& \lim_{t\to\infty} \int_t^{t+1} \|u(s)-M\|_2^2\ \mathrm{d}s = \lim_{t\to\infty} \int_t^{t+1} \|v(s)\|_{H^2}^2\ \mathrm{d}s = 0\,, \label{cvi}
	\end{align}
	where $P(t) := \mathcal{K}[u(t)-M]$ for $t\ge 0$. 
\end{theorem}

As already mentioned, Theorem~\ref{thm1} supplements previous results in the literature showing the long term convergence of $(u-M,v)$ to zero, either in low space dimension $n\in \{1,2,3\}$, see~\cite[Theorem~1.2]{LiWi2023a}, or when $\|v^{in}\|_\infty$ is sufficiently small, see~\cite[Theorem~1.2]{LiZh2021}. As in~\cite{LiWi2023a}, the proof of Theorem~\ref{thm1} relies on the so-called duality estimate derived from~\eqref{ks1} (Lemma~\ref{lem4}) and the dissipativity properties of~\eqref{ks2} (Lemma~\ref{lem5}). The building block of the proof is to show that $\|\nabla P\|_2^2 + a \|v\|_2^2$  is a Liapunov functional for~\eqref{ks} for a suitable choice of~$a>0$. This step is the main difference with the approach developed in~\cite{LiWi2023a} where a functional of the form $\|\nabla P\|_2^2 + b \|\nabla v\|_2^2$ with $b>0$ is used.

\begin{remark}
	When $\gamma(0)=0$, Theorem~\ref{thm1} is no longer true and convergence of $u(t)$ as $t\to\infty$ to a non-constant limit may take place, see \cite[Theorem~1.5]{Wink2023b}.
\end{remark}

We do not address here the issue of the existence of global solutions to~\eqref{ks} and refer to \cite{LiWi2023b, LiZh2021, Wink2023b} for the existence of global bounded classical solutions and to \cite{LiWi2023a, LiWi2023b, LiZh2021, Wink2022a, Wink2023a} for that of global weak solutions under various assumptions on $\gamma$ (with either $\gamma(0)=0$ or $\gamma(0)>0$) and the space dimension~$n$. In particular, given a global weak solution $(u,v)$ to~\eqref{ks} constructed in~\cite{LiWi2023a, LiWi2023b} and $t_0>0$, $(t,x)\mapsto (u,v)(t+t_0,x)$ is a weak solution to~\eqref{ks} in the sense of Definition~\ref{def2}, so that the convergence stated in Theorem~\ref{thm1} applies to these solutions. 

We next combine~\eqref{cvi} and the time monotonicity of the $L^\infty$-norm of $v$ to supplement the convergence~\eqref{cvp} of~$v$ in $H^1(\Omega)$ with convergence to zero of $v$ in $L^\infty(\Omega)$ when the space dimension $n$ ranges in $\{1,2,3\}$, thereby recovering the outcome of \cite[Theorem~1.2]{LiWi2023a}.

\begin{corollary}\label{cor1}
	Assume that $n\in\{1,2,3\}$. Under the assumptions of Theorem~\ref{thm1}, one has also
	\begin{equation*}
		\lim_{t\to \infty} \|v(t)\|_{\infty} = 0\,.
	\end{equation*}
\end{corollary}

\section{Proofs}\label{sec2}

We begin with the definition of a global weak solution to~\eqref{ks} and introduce the Hilbert space
\begin{equation*}
	H_N^2(\Omega) := \{ z\in H^2(\Omega)\ :\ \nabla z\cdot\mathbf{n} = 0 \;\;\text{ on }\;\; \partial\Omega \}\,, 
\end{equation*}
which is actually the domain of the Laplace operator in $L^2(\Omega)$ supplemented with homogeneous Neumann boundary conditions.

\begin{definition}\label{def2}
	Consider $u^{in}\in L_+^1(\Omega)\cap H^1(\Omega)'$ and $v^{in} \in L_+^\infty(\Omega)\cap H^1(\Omega)$. A global weak solution to~\eqref{ks} is a couple of non-negative functions
	\begin{equation*}
		(u,v)\in C_w([0,\infty),H^1(\Omega)') \times C([0,\infty),L^2(\Omega))
	\end{equation*}
	satisfying, for any $t>0$, 
	\begin{align*}
		u & \in L^2((0,t)\times \Omega)\,, \\
		v & \in L^\infty((0,t)\times\Omega) \cap W^{1,2}((0,t),L^2(\Omega)) \cap L^2((0,t),H_N^2(\Omega))\,,
	\end{align*}
	along with
	\begin{subequations}\label{d1}
		\begin{equation}
			\int_\Omega u(t) \vartheta(t)\ \mathrm{d}x - \int_\Omega u^{in} \vartheta(0)\ \mathrm{d}x = \int_0^t \int_\Omega \big[ u \gamma(v) \Delta\vartheta + u \partial_t\vartheta \big]\ \mathrm{d}x\mathrm{d}s \label{d1a}
		\end{equation}
	for all $\vartheta\in W^{1,2}((0,t),L^2(\Omega)) \cap L^2((0,t),H_N^2(\Omega))$ and
	\begin{equation}
		\begin{split}
		\partial_t v - \Delta v + uv & = 0 \;\;\text{ a.e. in }\;\; (0,t)\times\Omega\,, \\
		\nabla v\cdot \mathbf{n} & = 0 \;\;\text{ a.e. on }\;\; (0,t)\times\partial\Omega\,.
		\end{split} \label{d1b}
	\end{equation}
	\end{subequations}
\end{definition}

We recall that, given a Banach space $X$ and $T\in (0,\infty]$, $C_w([0,T),X)$ denotes the space of weakly continuous functions from $[0,T)$ to $X$.

We next derive several estimates on $u$ and $v$ which are already well-known, see \cite{LiWi2023a}. From now on, $(c_i)_{i\ge 1}$ denote positive constants depending only on $\Omega$, $\gamma$ in~\eqref{n3}, $u^{in}$, and $v^{in}$.

\begin{lemma}\label{lem3}
	For $t\ge 0$, 
	\begin{equation}
		\langle u(t) \rangle = M = \langle u^{in} \rangle \;\;\text{ and }\;\; \|v(t)\|_\infty \le V := \|v^{in}\|_\infty\,. \label{b1}
	\end{equation}
\end{lemma}

\begin{proof}
	Lemma~\ref{lem3} readily follows from~\eqref{d1a} (with $\vartheta\equiv 1$), along with~\eqref{d1b}, the non-negativity of $uv$, and the comparison principle.
\end{proof}

We next exploit the specific form of~\eqref{ks1} to derive a so-called duality estimate on~$u$.

\begin{lemma}\label{lem4}
	Set $P=\mathcal{K}[u-M]$. Then, for $t>0$, 
	\begin{equation*}
		P \in  W^{1,2}((0,t),L^2(\Omega)) \cap L^2((0,t),H_N^2(\Omega))
	\end{equation*}
	and
	\begin{equation}
		\frac{\mathrm{d}}{\mathrm{d}t} \|\nabla P\|_2^2 = - 2 \int_\Omega \gamma(v) (u-M)^2\ \mathrm{d}x - 2M \int_\Omega \gamma(v) (u-M)\ \mathrm{d}x\,. \label{d2}
	\end{equation}
\end{lemma}

\begin{proof}
As
\begin{equation}
	\partial_t P = \langle u \gamma(v) \rangle - u \gamma(v) \;\;\text{ in }\;\; (0,\infty)\times\Omega \label{eqP}
\end{equation}
by~\eqref{n2a} and~\eqref{d1a} (with a suitable choice of test functions), the claimed regularity of~$P$ follows from~\eqref{n3}, \eqref{n2}, the square integrability of~$u$, and the boundedness of~$v$. It then follows from~\eqref{eqP} that
\begin{align*}
	\frac{1}{2}\frac{\mathrm{d}}{\mathrm{d}t} \|\nabla P\|_2^2 & = - \int_\Omega \partial_t P \Delta P\ \mathrm{d}x = \int_\Omega (u-M) \big[ \langle u \gamma(v) \rangle - u\gamma(v) \big]\ \mathrm{d}x \\
	& = - \int_\Omega \gamma(v) u(u-M)\ \mathrm{d}x \\
	& = - \int_\Omega \gamma(v) (u-M)^2\ \mathrm{d}x - M \int_\Omega \gamma(v) (u-M)\ \mathrm{d}x\,,
\end{align*}
which completes the proof.
\end{proof}

We next take advantage of the non-positivity of the right hand side of~\eqref{ks2} to obtain a classical energy estimate on $v$.

\begin{lemma}\label{lem5}
	For $t\ge 0$,
	\begin{equation*}
		\frac{\mathrm{d}}{\mathrm{d}t} \|v\|_2^2 + 2 \|\nabla v\|_2^2 + 2 \|v\sqrt{u}\|_2^2 = 0\,. 
	\end{equation*}
\end{lemma}

At this point, we deviate from the proof of~\cite[Theorem~1.2]{LiWi2023a} and construct a Liapunov functional associated with~\eqref{ks}, building upon the outcome of Lemma~\ref{lem4} and Lemma~\ref{lem5}. This is clearly the main building block of the proof. As a preliminary step, we observe that the continuity and positivity~\eqref{n3} of~$\gamma$ and the boundedness~\eqref{b1} of~$v$ imply that
\begin{equation}
	\gamma_* := \min_{s\in [0,V]}\{\gamma(s)\}>0\,. \label{lb}
\end{equation}
We also recall the Poincar\'e-Wirtinger inequality: there is $c_1>0$ such that
\begin{equation}
	\big\| z - \langle z \rangle \big\|_2 \le c_1 \|\nabla z\|_2\,, \qquad z\in H^1(\Omega)\,. \label{pwi}
\end{equation}

\begin{proposition}\label{prop6}
	There is $c_2>0$ such that, for $t\ge 0$,
	\begin{equation*}
		\frac{\mathrm{d}}{\mathrm{d}t} \Big( \|\nabla P\|_2^2 + c_2 \|v\|_2^2 \Big) + \gamma_* \|u-M\|_2^2 + c_2 \|\nabla v\|_2^2 \le 0\,. 
	\end{equation*}
\end{proposition}

\begin{proof}
	It follows from~\eqref{n3}, \eqref{b1}, \eqref{pwi}, and H\"older's inequality that
	\begin{align}
		2M \int_\Omega \gamma(v) (u-M)\ \mathrm{d}x & = 2M \int_\Omega \big[ \gamma(v) - \gamma(\langle v \rangle) \big] (u-M)\ \mathrm{d}x \nonumber\\
		& \le 2M \|\gamma'\|_{L^\infty(0,V)} \int_\Omega |v-\langle v\rangle| |u-M|\ \mathrm{d}x \nonumber\\
		& \le 2M \|\gamma'\|_{L^\infty(0,V)} \|v-\langle v\rangle\|_2 \|u-M\|_2 \nonumber\\
		& \le 2Mc_1 \|\gamma'\|_{L^\infty(0,V)} \|\nabla v\|_2 \|u-M\|_2\,. \label{e1}
	\end{align}
Setting $c_2 := \big( Mc_1 \|\gamma'\|_{L^\infty(0,V)} \big)^2/\gamma_*$, we infer from Lemma~\ref{lem4}, Lemma~\ref{lem5}, \eqref{lb}, \eqref{e1}, and Young's inequality that
\begin{align*}
	\frac{\mathrm{d}}{\mathrm{d}t} \Big( \|\nabla P\|_2^2 + c_2 \|v\|_2^2 \Big) & \le - 2 \gamma_* \| u-M\|_2^2 - 2 c_2 \|\nabla v\|_2^2 \\
	& \qquad + 2Mc_1 \|\gamma'\|_{L^\infty(0,V)} \|\nabla v\|_2 \|u-M\|_2 \\
	& \le - 2 \gamma_* \| u-M\|_2^2 - 2 c_2 \|\nabla v\|_2^2 \\
	& \qquad + \gamma_* \|u-M\|_2^2 + \frac{M^2 c_1^2 \|\gamma'\|_{L^\infty(0,V)}^2}{\gamma_*} \|\nabla v\|_2^2 \\
	& = - \gamma_* \| u-M\|_2^2 - c_2 \|\nabla v\|_2^2\,,
\end{align*}
as claimed.
\end{proof}

We next argue as in~\cite[Lemma~3.2]{LiWi2023a} to obtain additional information on~$v$.

\begin{lemma}\label{lem7}
	For $t\ge 0$, 
	\begin{equation*}
		\frac{\mathrm{d}}{\mathrm{d}t} \|\nabla v\|_2^2 + 2M \|\nabla v\|_2^2 + \|\Delta v\|_2^2 \le V^2 \|u-M\|_2^2\,.
	\end{equation*}
\end{lemma}

\begin{proof}
	We infer from~\eqref{d1b}, \eqref{b1}, and H\"older's and Young's inequalities that
	\begin{align*}
		\frac{1}{2} \frac{\mathrm{d}}{\mathrm{d}t} \|\nabla v\|_2^2 + \|\Delta v\|_2^2 & = \int_\Omega u v\Delta v\ \mathrm{d}x \\
		& = \int_\Omega (u-M) v\Delta v\ \mathrm{d}x + M \int_\Omega v\Delta v\ \mathrm{d}x \\
		& \le V \|u-M\|_2 \|\Delta v\|_2 - M \|\nabla v\|_2^2 \\
		& \le \frac{\|\Delta v\|_2^2}{2} + \frac{V^2 \|u-M\|_2^2}{2} - M \|\nabla v\|_2^2\,,
	\end{align*}
	from which Lemma~\ref{lem7} follows.
\end{proof}

Summarizing the outcome of Proposition~\ref{prop6} and Lemma~\ref{lem7}, we have so far obtained the following estimates on $u$ and $v$.

\begin{proposition}\label{prop8}
	There is $c_3>0$ such that
	\begin{align}
		& \|P(t) \|_{H^1} + \|v(t)\|_{H^1} \le c_3\,, \qquad t\ge 0\,, \label{b2} \\
		& \int_0^\infty \Big[ \|u(s)-M\|_2^2 + \|\nabla v(s)\|_2^2 + \|\Delta v(s)\|_2^2 \Big]\ \mathrm{d}s \le c_3\,, \label{b3}
	\end{align}
and
\begin{equation}
	\lim_{t\to\infty} \|\nabla v(t)\|_2 = 0\,. \label{b4}
\end{equation}
\end{proposition}

\begin{proof}
	The bounds~\eqref{b2} and~\eqref{b3} being immediate consequences of~\eqref{n2b}, \eqref{pwi}, Proposition~\ref{prop6} and Lemma~\ref{lem7}, we are left with proving~\eqref{b4}. To this end, we deduce from Lemma~\ref{lem7} that, for $t\ge 0$, 
	\begin{equation}
		\|\nabla v(t)\|_2^2 \le \|\nabla v^{in}\|_2^2 e^{-2Mt} + V^2 \int_0^t e^{2M(s-t)} \|u(s)-M\|_2^2\ \mathrm{d}s\,. \label{g1}
	\end{equation}
At this point, we recall that, if $F$ belongs to $L^1(0,\infty)$, then a straightforward consequence of the Lebesgue dominated convergence theorem is that
\begin{equation}
	\lim_{t\to\infty} \int_0^t e^{\alpha(s-t)} |F(s)|\ \mathrm{d}s = 0 \label{g2}
\end{equation}
for any $\alpha>0$. Thanks to~\eqref{b3}, $s\mapsto \|u(s)-M\|_2^2$ belongs to $L^1(0,\infty)$ and we use~\eqref{g2} (with $\alpha=2M$ and $F=\|u-M\|_2^2$) to take the limit $t\to\infty$ in~\eqref{g1} and obtain~\eqref{b4}, thereby completing the proof.
\end{proof}

The final step of the proof of Theorem~\ref{thm1} deals with the convergence of $\| v(t)\|_1$ and $\|\nabla P(t)\|_2$ as $t\to\infty$.

\begin{lemma}\label{lem9}
	\begin{equation*}
		\lim_{t\to\infty} \| v(t)\|_1 = \lim_{t\to\infty} \|\nabla P(t)\|_2 = 0\,.
	\end{equation*}
\end{lemma}

\begin{proof}
	Let us begin with the convergence of $\| v\|_1$ and infer from~\eqref{d1b}, \eqref{b1}, and the non-negativity of $v$ that, for $t\ge 0$, 
	\begin{align*}
		\frac{\mathrm{d}}{\mathrm{d}t} \|v\|_1 & = - \int_\Omega u v\ \mathrm{d}x = - \int_\Omega (u-M) v\ \mathrm{d}x - M \|v\|_1 \\
		& = - \int_\Omega (u-M) (v-\langle v\rangle)\ \mathrm{d}x - M \|v\|_1\,.
	\end{align*} 
Hence, by~\eqref{pwi} and H\"older's inequality,
\begin{equation*}
	\frac{\mathrm{d}}{\mathrm{d}t} \|v\|_1 + M \|v\|_1 \le \|u-M\|_2 \|v-\langle v\rangle\|_2 \le c_1 \|u-M\|_2 \|\nabla v\|_2\,.
\end{equation*}
We then integrate with respect to time to find
\begin{equation}
	\|v(t)\|_1 \le \|v^{in}\|_1 e^{-Mt} + c_1 \int_0^t e^{M(s-t)} \|u(s)-M\|_2 \|\nabla v(s)\|_2\ \mathrm{d}s\,. \label{g3}
\end{equation}
Since $s\mapsto \|u(s)-M\|_2 \|\nabla v(s)\|_2$ belongs to $L^1(0,\infty)$ by~\eqref{b3}, we deduce from~\eqref{g2} (with $\alpha=M$ and $F=\|u-M\|_2 \|\nabla v\|_2$) that the right hand side of~\eqref{g3} converges to zero as $t\to\infty$. Consequently, 
\begin{equation*}
	\lim_{t\to\infty} \|v(t)\|_1 = 0\,.
\end{equation*}

Similarly, we infer from~\eqref{d2}, \eqref{lb}, and~\eqref{e1} that, for $t\ge 0$, 
\begin{equation*}
	\frac{\mathrm{d}}{\mathrm{d}t} \|\nabla P\|_2^2 \le -2 \gamma_* \|u-M\|_2^2 + 2M c_1 \|\gamma'\|_{L^\infty(0,V)} \|\nabla v\|_2\|u-M\|_2\,.
\end{equation*}
Moreover, by~\eqref{n2}, \eqref{pwi}, and H\"older's inequality,
\begin{align*}
	\|\nabla P\|_2^2 & =  - \int_\Omega P \Delta P\ \mathrm{d}x = \int_\Omega (u-M) P\ \mathrm{d}x \\
	& \le \|u-M\|_2 \|P\|_2 \le c_1 \|u-M\|_2 \|\nabla P\|_2\,,
\end{align*}
so that
\begin{equation*}
	\|\nabla P\|_2 \le c_1 \|u-M\|_2\,.
\end{equation*}
Gathering the above inequalities and setting $c_4 := 2\gamma_*/c_1^2$ and $c_5 := 2M c_1\|\gamma'\|_{L^\infty(0,V)}$, we end up with
\begin{equation*}
	\frac{\mathrm{d}}{\mathrm{d}t} \|\nabla P\|_2^2 + c_4 \|\nabla P\|_2^2 \le c_5 \|\nabla v\|_2 \|u-M\|_2\,.
\end{equation*} 
Hence, after integration with respect to time,
\begin{equation*}
	\|\nabla P(t)\|_2^2 \le \|\nabla P(0)\|_2^2 e^{-c_4 t} + c_5 \int_0^t e^{c_4(s-t)} \|\nabla v(s)\|_2 \|u(s)-M\|_2\ \mathrm{d}s\,,
\end{equation*}
and we argue as above to conclude that $\|\nabla P(t)\|_2$ converges to zero as $t\to\infty$. 	
\end{proof}

Theorem~\ref{thm1} is now an immediate consequence of Proposition~\ref{prop8} and Lemma~\ref{lem9}.

\begin{proof}[Proof of Theorem~\ref{thm1}]
	The convergences~\eqref{cvp} follow from~\eqref{pwi}, \eqref{b4}, and Lemma~\ref{lem9}, while the time integrability~\eqref{b3} of $\|u-M\|_2$ and $\|\Delta v\|_2$, along with~\eqref{cvp} and elliptic regularity, gives~\eqref{cvi}.
\end{proof}

We finally provide the proof of Corollary~\ref{cor1}.

\begin{proof}[Proof of Corollary~\ref{cor1}]
Since $n\in\{1,2,3\}$, the space $H^2(\Omega)$ is continuously embedded in $L^\infty(\Omega)$ and we deduce from~\eqref{cvi} that
\begin{equation*}
	\lim_{t\to\infty} \int_t^{t+1} \|v(s)\|_\infty^2\ \mathrm{d}s = 0\,.
\end{equation*}
We then infer from~\eqref{ks2}, \eqref{ks3}, the non--negativity of~$u$ and~$v$, and the comparison principle that 
\begin{equation*}
	\|v(t+1)\|_\infty \le \|v(s)\|_\infty\,, \qquad s\in [t,t+1]\,.
\end{equation*}
Combining the previous two properties readily gives
\begin{equation*}
	\lim_{t\to\infty} \|v(t)\|_\infty^2 = \lim_{t\to\infty} \|v(t+1)\|_\infty^2 \le \lim_{t\to\infty} \int_t^{t+1} \|v(s)\|_\infty^2\ \mathrm{d}s = 0\,,
\end{equation*}
which completes the proof.
\end{proof}


\bibliographystyle{siam}
\bibliography{LTSHCMLSC}

\end{document}